\documentclass[11pt,reqno]{amsart}
\usepackage{amsmath,amssymb,amsthm,amsfonts,enumerate,hyperref, caption, verbatim}
\theoremstyle{plain}
\newtheorem{theorem}{Theorem}[section]
\newtheorem{theorem2}{Theorem}[section]
\newtheorem{prop}[theorem2]{Proposition}
\newtheorem{cor}[theorem]{Corollary}

\theoremstyle{definition}
\newtheorem{definition}[theorem2]{Definition}

\theoremstyle{remark}
\newtheorem*{remark}{Remark}

\numberwithin{equation}{section}

\begin{document}

\title{The canonical Einstein metric on $G_{2}$ is dynamically unstable under the Ricci flow}
\author{Stuart James Hall}
\begin{abstract}
In this note we show that the bi-invariant Einstein metric on the compact Lie group $G_{2}$ is dynamically unstable as a fixed point of the Ricci flow. This completes the stability analysis for the bi-invariant metrics on the compact, connected, simple Lie groups. Interestingly, $G_{2}$ is the only unstable exceptional group.
\end{abstract}
\maketitle
\section{Introduction} 
\label{intro}

\noindent On any compact, connected, simple Lie group $G$ there is a canonical inner product on the Lie algebra $\mathfrak{g}$ given by multiplying the Killing form by $-1$; more precisely,
$$\langle X,Y\rangle = -\mathrm{tr}(\mathrm{ad}(X)\circ\mathrm{ad}(Y)),$$  
where $X,Y\in \mathfrak{g}$ and $\mathrm{ad}:\mathfrak{g}\rightarrow\mathfrak{gl}(\mathfrak{g})$ is the adjoint representation. This induces a Riemannian metric $g$ on $G$ that is bi-invariant (both the actions of $G$ on itself on the left and the right  are isometries). Up to scale, such a bi-invariant metric is unique. The metric $g$ is an Einstein metric satisfying $Ric(g) = g/4$ and is thus considered a fixed point of the Ricci flow
$$
\frac{\partial g}{\partial t} = -2Ric(g),
$$
as the flow induces only homothetic scaling. A natural question is to consider whether a given Einstein metric is stable in the sense that any flow starting at a small perturbation returns to the Einstein metric. In this note we prove the following theorem (we refer the reader to Section \ref{sec:2} for precise definitions of stability):
\begin{theorem}\label{mainT}
	The bi-invariant Einstein metric on the compact, simple Lie group $G_{2}$ is dynamically unstable under the Ricci flow.
\end{theorem}
There has been a good deal of work on determining the stability of Einstein metrics and it is expected that stable geometries should be quite special. By combining the above result with the work of Cao and He \cite{CH}, we now have a complete understanding of Ricci flow stability for the canonical Einstein metrics on the compact, connected, simply-connected, simple Lie groups.
\begin{cor}\label{mainC}
	If $G$ is a compact, connected, simply-connected, simple Lie group and the bi-invariant metric is dynamically stable then $G$ is one of:
	\begin{itemize}
		\item $\mathrm{SU}(2)$, $\mathrm{Spin}(n)$ for $n\geq 7$, $E_{6}$, $E_{7}$, $E_{8}$, or $F_{4}$.
	\end{itemize}  
	If the metric is unstable then $G$ is one of:
	\begin{itemize}
		\item  $\mathrm{SU}(n)$ for $n\geq 3$, $\mathrm{Sp}(n)$ for $n\geq 2$, $G_{2}$.
	\end{itemize}
\end{cor}
The stability of all of the groups in Corollary \ref{mainC} \textit{except} $G_{2}$ can be determined by studying the variational  stability of Perelman's $\nu$ functional; this is what was achieved by Cao and He in \cite{CH}. The reason that the stability of $G_{2}$ is not accessible by such analysis is that it admits certain neutral directions of deformation coming from conformal variations of the metric. In fact, the conformal variations are built from the eigenfunctions associated to the first non-zero eigenvalue of the ordinary Laplacian.  Recently, Kr\"oncke proved a stability criterion that can be checked for such neutral conformal directions \cite{KKCAG}. The criterion simply involves integrating the cube of the eigenfunctions and determining if the resulting integral is non-zero (see Theorem \ref{prop:3}). Kr\"oncke applied this test to the Fubini--Study metric on $\mathbb{CP}^{n}$ for $n\geq 2$ and showed these metrics are unstable. There is also the recent work of Knopf and Sesum \cite{KS} where they independently prove Kr\"oncke's criterion and apply it to $\mathbb{CP}^{n}$. Using completely different methods for computing the integrals of eigenfunctions, Murphy, Waldron and the author have generalised the $\mathbb{CP}^{n}$ result to the K\"ahler--Einstein metric on all $k$-plane Grassmannians on $\mathbb{C}^{n}$ with $n\neq 2k$ \cite{HMW}.\\   
\\
The methods in this note are similar to those in \cite{HMW} and the proof of Theorem \ref{mainT} involves putting together various classical facts from the representation theory of compact Lie groups. The characters of irreducible representations are representatives of eigenvalues and thus are class functions. Such functions can be integrated using the Weyl integration formula which transfers the integral to one over a two-dimensional maximal torus.  This integral can then be evaluated explicitly and shown to be non-zero. Hence Kr\"oncke's result shows that the metric is unstable.\\
\\
Determining which geometries are stable under the Ricci flow is useful for establishing what possible singularity models could occur in higher dimensions. It would be interesting to consider what happens to the Ricci flow on $G_{2}$. For example, as $G_{2}$ is a principal $\mathrm{SU}(3)$-bundle over the sphere $\mathbb{S}^{6}$, one possibility is that the flow shrinks the $\mathrm{SU}(3)$-fibre and approaches the (stable) round metric on $\mathbb{S}^{6}$.  It is also interesting to note that $G_{2}$ features in the lowest-dimensional, non-spherical, stable Einstein metric that is currently known.  This is the symmetric space $G_{2}/SO(4)$ which was shown to be stable by Cao and He in \cite{CH}.\\
\\
\textit{Acknowledgements:}
It is a pleasure to thank Tommy Murphy and James Waldron for many useful conversations about the stability of symmetric spaces.  I would also like to thank David Stewart for his comments on the representation theory of $G_{2}$. I thank the referees for their careful reading of the paper and providing useful suggestions for improvement.

\section{Background}\label{sec:2}
\subsection{Stability of Einstein metrics under the Ricci flow}
We will take the following definition of dynamical stability.

\begin{definition}[(Dynamical stability of Einstein metrics)]\label{Def 1}
	Let $(M^{n},g_{E})$ be a compact Einstein manifold. The metric $g_{E}$ is said to be \textit{dynamically stable} for the Ricci flow if for any $m\geq 3$ and any $C^{m}$-neighbourhood $U$ of $g_{E}$ in the space of sections $\Gamma(s^{2}(TM^{\ast}))$, there exists a $C^{m+2}$ neighbourhood of $g_{E}$, $V \subset U$, such that: 
	\begin{enumerate}
		\item for any $g_{0}\in V$, the volume normalised Ricci flow 
		$$ \dfrac{\partial g}{\partial t} =-2Ric(g)+\frac{2}{n\mathrm{Vol(g(t))}}\left(\int_{M}\textrm{scal}(g)dV_{g}\right)g,$$
		with $g(0) = g_{0}$ exists for all time,
		\item the metrics $g(t)$ converge modulo diffeomorphism to an Einstein metric in $U$.
	\end{enumerate}
\end{definition}
\begin{remark}
The notion of stability given in Definition \ref{Def 1} is sometimes referred to as \textit{ asymptotic stability} (for example in \cite{CCGG}). 	
\end{remark}	

The literature on stability is quite extensive; we refer the reader to the references for the details of the theory we will use. There is also a survey of the field in chapter 35 of \cite{CCGG}.  One key tool is the analysis of Perelman's $\nu$ functional which was introduced in \cite{Per1}. An Einstein metric is a critical point of $\nu$ and, as $\nu$ is monotonically increasing along the Ricci flow except at its critical points, the stability of the flow can be investigated by computing the second variation of $\nu$ at an Einstein metric. This procedure was first carried out by Cao, Hamilton and Ilmanen in \cite{CHI} who gave a formula for the second variation of the $\nu$ functional (the proof of the variational formula was given later in \cite{CZ}). They demonstrated that understanding the classical linear stability of a compact Einstein metric is the same as understanding the variational stability of the $\nu$ functional. For a compact Einstein manifold $(M,g)$, the space of sections of symmetric two-tensors admits the following $L^{2}$-orthogonal decomposition, 
$$\Gamma(s^{2}(TM^{\ast})) = \mathrm{Ker}(\mathrm{div})\oplus \mathrm{Im}(\mathrm{div}^{\ast}),$$
where ${\mathrm{div}:\Gamma(s^{2}(TM^{\ast}))\rightarrow \Omega^{1}(M)}$ is the divergence map. The kernel $\mathrm{Ker}(\mathrm{div})$ admits a further decomposition,
$$\mathrm{Ker}(\mathrm{div}) = K_{0}\oplus \mathbb{R}g,$$
where 
$$K_{0} = \left\{h\in \Gamma(s^{2}(TM^{\ast})) \ | \ \mathrm{div}(h)=0 \textrm{ and } \int_{M}\mathrm{tr}(h)dV_{g} = 0\right\}.$$
The Lichnerowicz Laplacian $\Delta_{L}$ preserves the subspace $K_{0}$ and the stability can be described  in terms of the spectrum of the Lichnerowicz Laplacian viewed as a map ${\Delta_{L}:K_{0}\rightarrow K_{0}}$. We will use the convention that the Lichnerowicz Laplacian has only finitely many positive eigenvalues and $a<\Delta_{L}< b$ means that the inequality is satisfied for the largest eigenvalue of $\Delta_{L}$ when restricted to the subspace $K_{0}$. The following result, which paraphrases the results contained in \cite{CHI}, could be taken to be the definition of linear stability.
\begin{prop}[(Cao--Hamilton--Ilmanen \cite{CHI})]\label{prop:1}
	Let $(M,g_{E})$ be a compact Einstein manifold satisfying $Ric(g_{E}) = \Lambda g_{E}$ with $\Lambda>0$. Then
	\begin{enumerate}
		\item $g_{E}$ is \textit{linearly stable} if $\Delta_{L}<-2\Lambda$,
		\item $g_{E}$ is \textit{linearly unstable} if $\Delta_{L}>-2\Lambda$.	
	\end{enumerate} 
\end{prop} 
In considering the case when the largest eigenvalue of $\Delta_{L}$ on $K_{0}$ is equal to $-2\Lambda$, we make the following definition.
\begin{definition}[(Infinitesimal solitonic deformation (cf. Definition 6.1 in \cite{KKCVP}))] \label{Def2}
If the largest eigenvalue of $\Delta_{L}$ on $K_{0}$ is exactly $-2\Lambda$ then we say that the Einstein metric is \textit{neutrally linearly stable}. We will refer to the corresponding eigentensors as \textit{infinitesimal solitonic deformations} of the metric $g_{E}$.
\end{definition}
Cao and He, building on the work in \cite{CHI}, investigated deformations coming from conformal variations of the Einstein metric.  In \cite{CH} they proved that, apart from when the manifold $(M,g_{E})$ is the round sphere, there is a linear map $\mathcal{S}:C^{\infty}_{0}(M)\rightarrow K_{0}$, where
$$
C^{\infty}_{0}(M) = \left\{f\in C^{\infty}(M) \ | \ \int_{M}f dV_{g}=0\right\},
$$
such that $\Delta_{L}(S(u)) = S(\Delta u)$ for all $u\in C^{\infty}_{0}(M)$. The map sends eigenvalues of the ordinary Laplacian to eigentensors of  the Lichnerowicz Laplacian and clearly preserves the eigenvalue; in particular, $\Delta_{L}\geq \Delta$. Cao and He gave a comprehensive study of the linear stability of compact symmetric spaces and showed that $G_{2}$ is neutrally linearly stable with the infinitesimal solitonic deformations given by $\mathcal{S}(\varphi)$, where $\varphi$ is an eigenfunction for the first non-zero eigenvalue of the Laplacian.\vspace{-1pt}\\
\\
The relationship between Cao--Hamilton--Ilmanen's linear stability and dynamical stability for Einstein metrics with $\Lambda>0$  was made precise by Kr\"oncke \cite{KKCVP} who built on earlier work by Sesum \cite{Sesum} and Haslhofer and M\"uller \cite{HasMul}. 
\begin{prop}[(Kr\"oncke (cf. Theorem 6.5 and Corollary 6.7 in \cite{KKCVP}))] \label{prop:2}
	Let $(M,g_{E})$ be a compact Einstein manifold satisfying $Ric(g_{E}) = \Lambda g_{E}$ with $\Lambda>0$.  Then 
	\begin{enumerate}
		\item $g_{E}$ is dynamically stable if it is linearly stable,
		\item $g_{E}$ is dynamically unstable if it is linearly unstable.	
	\end{enumerate}
\end{prop}
\begin{remark}
In \cite{KKCVP} the results are stated with a further condition that $g_{E}$ has to satisfy, namely that the infinitesimal solitonic deformations must be integrable (i.e. the deformation comes from a genuine curve of Ricci solitons). As our notion of linear stability from Proposition \ref{prop:1} does not include the possibility that the Lichnerowicz Laplacian restricted to $K_{0}$ has an eigenvalue equal to $-2\Lambda$, $g_{E}$ has no such deformations and there is nothing to check so we omit this condition from the statement of Proposition \ref{prop:2}. The fact that linear instability implies dynamical instability does not require any integrability assumption on  
the infinitesimal solitonic deformations. In the case that the Einstein metric is neutrally linearly stable, one can conclude that the metric is dynamically stable if one can show that all infinitesimal solitonic deformations are integrable.
\end{remark}
For Einstein metrics with infinitesimal solitonic deformations corresponding to $\mathcal{S}(\varphi)$ where $\varphi$ an eigenfunction of the Laplacian, Kr\"oncke proved the following result giving a criterion for demonstrating dynamical instability.
\begin{theorem}[(Kr\"oncke (cf. Theorem 1.7 in \cite{KKCAG}))]\label{prop:3}
	Let $(M,g_{E})$ be an Einstein metric satisfying $Ric(g_{E}) = \Lambda g_{E}$ for $\Lambda>0$ and  suppose that  $\varphi$ satsfies $\Delta \varphi = -2\Lambda \varphi$. If 
	$$\int_{M}\varphi^{3} dV_{g_{E}} \neq 0,$$
	then $g_{E}$ is dynamically unstable.	
\end{theorem}
\begin{remark}
A necessary condition for the infinitesimal solitonic deformation $\mathcal{S}(\varphi)$ to be integrable is that the integral in Theorem \ref{prop:3} vanishes. Hence, in the course of the proof of Theorem \ref{mainT}, we show that $G_{2}$ has non-integrable infinitesimal solitonic deformations. This can be compared with Theorem 1.5 in \cite{KKCVP} where it is shown that the Fubini--Study metric on  $\mathbb{CP}^{n}$ has non-integrable solitonic deformations for $n\geq 2$. 	
\end{remark}	
\subsection{The structure and geometry of the Lie group $G_{2}$}
We begin by recalling some facts about general compact, connected, simple Lie groups $G$ with the metric induced by the Killing form. It has long been known (e.g. \cite{Sug} and \cite{Ura}) that the eigenvalues of the Laplacian can be calculated using Freudenthal's formula. Corresponding to an irreducible representation with highest weight $\lambda$, $V_{\lambda}$, there is an eigenvalue of the Laplacian, $\mu_{\lambda}$, given by
\begin{equation}\label{Fform}
	\mu_{\lambda} =|\rho|^{2}- | \lambda+\rho|^{2},
\end{equation}
where $\rho=\frac{1}{2}\sum_{\alpha \in R^{+}}\alpha$ and $R^{+}$ is the set of positive roots. Furthermore, if $V_{\lambda}$ is the complexification of an irreducible real representation, the character, $\chi_{\lambda}$,  associated to $V_{\lambda}$ is real and is an eigenfunction of the Laplacian with eigenvalue $\mu_{\lambda}$ \cite{Ber}.\\
\\
In this note, $G_{2}$ denotes the $14$-dimensional compact real form of the Lie group with complex simple Lie algebra $\mathfrak{g}_{2}$. The rank of $G_{2}$ is 2 and we let $T$ be a maximal torus with Lie algebra $\mathfrak{t}$. We will use the description of the root system $R\subset\mathfrak{t}^{\ast} \cong \mathbb{R}^{2}$ given in \cite{FH}; there is a short simple root $\alpha_{1} = (1,0)$ and long root $\alpha_{2} = (-3/2,\sqrt{3}/2)$.  The positive roots are then
$$R^{+} = \{\alpha_{1}, 3\alpha_{1}+\alpha_{2}, 2\alpha_{1} + \alpha_{2}, 3\alpha_{1}+2\alpha_{2}, \alpha_{1}+\alpha_{2}, \alpha_{2} \}.$$
It is important to note that in \cite{FH} the authors scale the Killing form so that, with respect to the metric induced upon the dual space, $|\alpha_{1}|^{2}=1$.  On $G_{2}$, this normalisation corresponds to scaling the metric $g\rightarrow \frac{1}{12}g$; in this case, the Einstein constant of the rescaled metric is $3$. The weight lattice is generated by the simple roots $\alpha_{1}$  and $\alpha_{2}$.  The intersection of the closed Weyl chamber and the weight lattice is generated by the fundamental weights $\omega_{1}$, $\omega_{2}$ where
$$\omega_{1} = 2\alpha_{1}+\alpha_{2} \textrm{ and } \omega_{2} = 3\alpha_{1} + 2\alpha_{2}.$$
The irreducible representation with highest weight $\lambda = a\omega_{1}+b\omega_{2}$ is denoted $\Gamma_{a,b}$ where ${a,b\in \mathbb{Z}_{+}}$.
\begin{prop}\label{G2evalvec}
	Let $g$ be the bi-invariant metric on $G_{2}$ induced by the Killing form. Then:
	\begin{enumerate}
		\item The largest eigenvalue of the Laplacian is $-1/2$ and this corresponds to the irreducible representation $\Gamma_{1,0}$.
		\item The character $\chi_{1,0}$ is a corresponding eigenfunction and can be represented in coordinates $\theta_{1},\theta_{2}$ on a maximal torus $T$ by the function
		$$\chi_{1,0}(\theta_{1},\theta_{2}) = 2\cos(\theta_{1})+2\cos(\theta_{2})+2\cos(\theta_{1}+\theta_{2})+1. $$
	\end{enumerate}
\end{prop}
\begin{proof}
	The smallest-dimensional non-trivial representation of $G_{2}$ has dimension 7 and corresponds to $\Gamma_{1,0}$ (i.e. it is associated to the smallest of the fundamental weights). The complex irreducible representation is a complexification of an underlying irreducible real representation (this representation comes from a description of $G_{2}$ as a certain subgroup of $SO(7)$ and so one can restrict the natural representation). \\
	\\
	Noting that $\rho = \omega_{1} + \omega_{2}$ we compute using Equation (\ref{Fform})
	$$\mu_{1,0} =   |\omega_{1}+\omega_{2}|^{2}-|2\omega_{1}+\omega_{2}|^{2}  = -6.$$
	If we scale any Riemannian metric by ${g\rightarrow kg}$ for some $k>0$, then the Einstein constant scales by ${\Lambda \rightarrow k^{-1}\Lambda}$ and eigenvalues of the Laplacian also scale by ${\mu\rightarrow k^{-1}\mu}$.  
	As mentioned, we need to rescale by taking $k=12$ and so the eigenvalue for the metric induced by the Killing form is $-1/2$ and we have proved part $(1)$.\\
	\\
	To compute the character we note the following formula computed in Proposition 24.48 of \cite{FH}
	\begin{equation}\label{WCF}
		\chi_{a,b} = \frac{S_{(a+2b+1,a+b+1)}-S_{(a+2b+1,b)}}{S_{(1,1)}-S_{(1)}},
	\end{equation}
	where $S_{(s,t)}(x_{1},x_{2},x_{3})$ is the Schur polynomial corresponding to the partition $(s,t)$ and the variables $x_{1},x_{2}$ and $x_{3}$ satisfy $\prod x_{i}=1$.  Here the character is defined on the actual torus $T$ and so $x_{1} = e^{i\theta_{1}}$ and $x_{2}=e^{i\theta_{2}}$.  Part 2 follows from substitution of $a=1$ and $b=0$ into Equation (\ref{WCF}) and simplifying. 
\end{proof}

\section{The proof of Theorem \ref{mainT} }
In order to compute the integral in Kr\"oncke's stability criterion, we need the Weyl integration formula.  Here we give the version from \cite{Bump}; to state it, we recall that a maximal torus $T$ with Lie algebra $\mathfrak{t}$ yields an $Ad$-invariant decomposition ${\mathfrak{g} = \mathfrak{t}\oplus \mathfrak{p}}$.   
\begin{prop}[(Weyl integration formula, Theorem 17.2 in \cite{Bump})]
	Let $G$ be a compact connected Lie group, $T$ a maximal torus, and $\mathfrak{p}$ as previously. If $f$ is a class function, and if $dg$ and $dt$ are Haar measures on $G$ and $T$ (normalized so that $G$ and $T$ have volume $1$), then
	\begin{equation}\label{WeylIF}
	\int_{G}f(g) \mathrm{d}g = \frac{1}{|W|}\int_{T}f(t)\det\left([Ad(t^{-1})-Id]|_{\mathfrak{p}} \right) \mathrm{d}t,
	\end{equation}
	where $W$ is the Weyl group.
\end{prop}
\begin{remark}
	For a compact connected Lie group $G$, the Haar measure is, up to scale, induced by the volume form of the bi-invariant metric.  As the proof of the Theorem \ref{mainT} will simply require a certain integral is non-zero, we need not worry carefully about the rescaling of the measures to ensure unit volumes. 
\end{remark}
We also note a computational simplification of the Formula (\ref{WeylIF}) by considering the quantity
$$\delta(\theta) = \prod_{\alpha\in R^{+}}(e^{\alpha(\theta)/2}-e^{-\alpha(\theta)/2}),$$
where $\theta$ is shorthand for coordinates on the maximal torus $T$. We have the following identity
$$ 	\det\left([Ad(\exp(-\theta))-Id]|_{\mathfrak{p}}\right) = \delta\bar{\delta}.  $$
We note further that $\delta$ is also the denominator (sometimes denoted $A_{\rho}$) in the Weyl character formula.  We now give the proof of the main result.
\begin{proof}[Proof of Theorem~{\rm\ref{mainT}}] 
	In order to prove instability, we must compute the integral in Theorem \ref{prop:3} for a $-2\Lambda-$eigenfunction of the Laplacian $\varphi$ and show the integral does not vanish.  Proposition \ref{G2evalvec} yields that the character of the irreducible 7-dimensional representation of $G_{2}$, 
	$$\chi_{1,0} = 2\cos(\theta_{1})+2\cos(\theta_{2})+2\cos(\theta_{1}+\theta_{2})+1.$$ 
	is such a $-2\Lambda$-eigenfunction. Clearly any power of the character is a class function and so the integral can be computed using the Weyl integration formula (\ref{WeylIF}). As we are only interested in whether or not the integral is zero, we do not worry about scaling the Haar measure (induced from the bi-invariant volume form) to have unit mass. The Jacobian factor $\det\left([Ad(t^{-1})-Id]|_{\mathfrak{p}} \right)$ can be computed from the quantity $\delta$ which itself is computable directly from the root data of $\mathfrak{g}_{2}$.  We calculate
	\begin{eqnarray*}
	\delta = 2\cos(\theta_{1}+3\theta_{2})-2\cos(3\theta_{1}+\theta_{2}) +2\cos(2\theta_{1}-\theta_{2})-2\cos(\theta_{1}-2\theta_{2}) & \\
	+2\cos(3\theta_{1}+2\theta_{2})-2\cos(2\theta_{1}+3\theta_{2})&.
	\end{eqnarray*}
	Hence we find
	$$\int_{G_{2}}\varphi^{3}dV_{g}  = C\int_{0}^{2\pi}\int_{0}^{2\pi} \chi_{1,0}^{3} \ \delta^{2} \ d\theta_{1}d\theta_{2},$$
	where $C>0$ is some fixed scaling constant. The result follows from the fact (easily computed in Maple)
	$$\int_{0}^{2\pi}\int_{0}^{2\pi} \chi_{1,0}^{3} \ \delta^{2} \ d\theta_{1}d\theta_{2}  =48\pi^{2} \neq 0.$$
	
\end{proof}

\vspace{5pt}
   Stuart James Hall\\
   School of Mathematics and Statistics, Herschel Building, Newcastle University, Newcastle-upon-Tyne, NE1 7RU\\
   United Kingdom
   \email{stuart.hall@ncl.ac.uk\\
   
}

\end{document}